\documentclass[a4paper,11pt]{article}
\usepackage[utf8]{inputenc}
\usepackage{amsmath, amsthm}
\usepackage{esint}
\usepackage{amsfonts}
\usepackage{amssymb}
\usepackage{hyperref}
\usepackage{graphicx}
\usepackage[T1]{fontenc}
\usepackage[francais,english]{babel}
\usepackage{listings}
\usepackage{vmargin} 
\usepackage{dsfont}
\usepackage{color}

\usepackage[left]{lineno}

\usepackage[final]{pdfpages}
\numberwithin{equation}{section}

\newtheorem{de}{Definition}[section]
\newtheorem{thm}[de]{Theorem}
\newtheorem{cor}[de]{Corollary}
\newtheorem{prop}[de]{Proposition}
\newtheorem{lem}[de]{Lemma}
\newtheorem{rem}[de]{Remark}
\newtheorem{exa}[de]{Example}

\newcommand\E{\mathcal{E}}
\newcommand\F{\mathcal{F}}

\newcommand\G{\mathcal{G}}

\newcommand\W{\mathcal{W}}

\newcommand\M{\mathcal{M}}

\newcommand\grho{\boldsymbol{\rho}}
\newcommand\x{\boldsymbol{x}}
\newcommand\y{\boldsymbol{y}}

\newcommand\supp{\mathop{\rm supp}}

\renewcommand{\textbf}[1]{\begingroup\bfseries\mathversion{bold}#1\endgroup}

\DeclareMathOperator*{\dive}{div}

\DeclareMathOperator*{\Rn}{\mathbb{R}^\textit{n}}
\DeclareMathOperator*{\R}{\mathbb{R}}
\DeclareMathOperator*{\Pa}{\mathcal{P}}
\DeclareMathOperator*{\Paa}{\mathcal{P}^{ac}}

\title{Nonlinear systems coupled through multi-marginal transport problems}

\author{Maxime Laborde \thanks{\scriptsize Department of Mathematics and Statistics, McGill University, Montreal, CANADA (\href{mailto:maxime.laborde@mcgill.ca}{\tt maxime.laborde@mcgill.ca)}}}

\date{}

\begin{document}

\maketitle
\begin{abstract}
In this paper, we introduce a dynamical urban planning model. This leads us to study a system of nonlinear equations coupled through multi-marginal optimal transport problems. 
A simple case consists in solving two equations coupled through the solution to the Monge-Ampère equation. 
We show that the Wasserstein gradient flow theory provides a very good framework to solve this highly nonlinear system. 
At the end, an uniqueness result is presented in dimension one based on convexity arguments. 
\end{abstract}

%
%

\section{Introduction}
Recently, Kinderlehrer, Monsaingeon and Xu proposed in \cite{KMX} a gradient flow approach to solve the Poisson-Nernst-Planck system
$$ \left\{\begin{array}{l}
\partial_t \rho_1 -\alpha \Delta \rho_1^m - \dive(\rho_1 \nabla(V_1+ \varphi))=0,\\
\partial_t \rho_2 - \beta \Delta \rho_2^m - \dive(\rho_2 \nabla(V_2- \varphi))=0,\\
-\Delta \varphi = \rho_1-\rho_2.
\end{array}\right.$$
This system is, for instance, used to model ionic transport of sereval interacting species. Inspired by this work we are interested in a "nonlinear" version where species $\rho_1$ and $\rho_2$ are coupled through the Monge-Ampère equation instead of the Poisson equation,
\begin{eqnarray}
\label{part5-Monge-Ampere}
\left\{ \begin{array}{l}
\partial_t \rho_1 - \alpha \Delta \rho_1^m - \dive(\rho_1 \nabla (V_1 +\varphi))=0,\\
\partial_t \rho_2 - \beta \Delta \rho_2^m - \dive(\rho_2 \nabla ( V_2 + \varphi^c))=0,\\
\det(I-D^2\varphi)\rho_2(Id - \nabla \varphi) = \rho_1,
\end{array}\right.
\end{eqnarray}
where $\varphi^c$ is the c-transform of $\varphi$,  $\varphi^c(x) = \sup_{y} |x-y|^2 -\varphi(y)$ and $|x|^2-\varphi$ is convex. \\

This kind of systems can arise naturally in urban planning. In a series of works \cite{BS,BS1,CS,CE,CE1,S1,S2,S_grenoble} (non-exhaustive list), static models of urban planning were proposed. 
A simplified model consists in considering an urban area region $\Omega$ where residents and services, given by two probability densities on $\Omega$, $\rho_1$ and $\rho_2$, want to minimize a quantity, $\E(\rho_1,\rho_2)$, to reach an ideal organization in the city. 
The total cost has to take into account a transportation cost between residential areas and service areas, a congestion effect for residential areas due to the fact that the population does not want to live in very crowded area and, on the contrary, services want to be more concentrated in order to increase efficiency and decrease management costs. 
Particularly, the cost functional $\E$ can be taken as
\begin{equation}
\label{part5-energy static}
\E(\rho_1,\rho_2)= W_c(\rho_1,\rho_2) + \F(\rho_1) + \G(\rho_2),
\end{equation} 
where $W_c$ is the value of an optimal transport problem with cost $c$. 
Several interpretations may be given to this cost. For example, it might represent the gas cost paid by workers to reach services area and then workers want to live close to services in order to decrease car travel.
$\F$ is an internal energy given by a convex superlinear function $F$,
$$ \F(\rho):= \left\{ \begin{array}{ll}
\int_\Omega F(\rho(x)) \, dx & \text{  if } F(\rho) \in L^1(\Omega),\\
+\infty & \text{  otherwise. }
\end{array}\right.$$
Since $F$ is superlinear and convex, $\F$ can be rewritten as
$$ \F(\rho) = \int_\Omega \frac{F(\rho)}{\rho}\rho,$$
where $\rho \mapsto \frac{F(\rho)}{\rho}$ is a increasing function which can be seen as the unhapiness of a citizen when he lives in a place where the population density is $\rho$. 
Finally, $\G$ is on the form
$$ \G(\rho)= \iint_{\Omega \times \Omega} h(|x-y|) \,d\rho(x)d\rho(y),$$
where $h$ is an increasing function modeling interactions between different services. \\

However, since a city is constantly evolving, it seems natural to study how evolve $\rho_1$ and $\rho_2$ in time.
This leads to study the gradient flow of $\E$ in a Wasserstein product space. 
In the case where $c$ is the quadratic cost, at least formaly, we find a system on the form \eqref{part5-Monge-Ampere} where $\varphi$ is a Kantorovich potential of $W_2(\rho_1,\rho_2)$ which implies that it satisfies the Monge-Ampère equation
$$\det(I-D^2\varphi)\rho_2(Id - \nabla \varphi) = \rho_1 .$$
In this paper, we propose to investigate a generalization of problem \eqref{part5-Monge-Ampere}. 
We extend to more than two  populations, then the transport problem becomes a multi-marginal transport problem. 
In other hand, the cost that workers want to minimize is not the same as the one of services or firms. 
Indeed, they have to take into account the gas cost to reach their work whereas this cost is not relevant for services.
Thus it is natural to assume that each population wants to minimize a transport problem with its own cost.
Since the system is not a gradient flow anymore, we will use a semi-implicit JKO scheme introduced in \cite{DFF} to deal with these different costs. \\

The organization of the paper is the following. 
Section \ref{part5-section preliminaire} recalls results from Optimal Transport and Multi-Marginal Transport theories. 
In section \ref{part5-paper main result}, we specify our problem and state our main result. 
Section \ref{part5-paper semi-implicit scheme} is devoted to the demonstration of the existence of solutions for the evolution problem \eqref{part5-Monge-Ampere}. 
The proof is based on a semi-implicit JKO scheme and on an extension of the Aubin-Lions Lemma in order to obtain strong regularity.
 At the end in section \ref{part5-paper uniqueness}, by convexity arguments, we give a uniqueness result in dimension one for some class of functionals.

\section{Preliminaries}
\label{part5-section preliminaire}

In the sequel, $\Omega$ represents a smooth open bounded subset of $\Rn$. 
\subsection{Wasserstein space}
For a detailed exposition, we refer to reference textbooks \cite{V1,V2,AGS,S}. 
We denote $\mathcal{M}^+(\Omega)$ the set of nonnegative finite Radon measures on $\Omega$, $\Pa(\Omega)$ the space of probability measures on $\Omega$, and $\Paa(\Omega)$, the subset of $\Pa(\Omega)$ of probability measures on $\Omega$ absolutely continuous with respect to the Lebesgue measure. 
For all $\rho,\mu \in \Pa(\Omega)$, we denote $\Pi(\rho,\mu)$, the set of probability measures on $\Omega \times \Omega$ having $\rho$ and $\mu$ as first and second marginals, respectively. 
If $\gamma \in \Pi(\rho,\mu)$, then $\gamma$ is called a transport plan between $\rho$ and $\mu$. 
\\
For all $\rho,\mu \in \Pa(\Omega)$, we denote by $W_2(\rho,\mu)$ the Wasserstein distance between $\rho$ and $\mu$,
$$ W_2^2(\rho,\mu)= \inf\left\{ \iint_{\Omega \times \Omega} |x-y|^2 \,d\gamma(x,y) \, : \, \gamma \in \Pi(\rho,\mu)\right\}.$$
Since this optimal transportation problem is a linear problem under linear constraint, it admits a dual formulation given by
$$ W_2^2(\rho,\mu)=\sup\left\{ \int_\Omega \varphi(x) \,d\rho(x) +\int_\Omega \psi(y) \, d\mu(y) \, : \, \varphi(x) +\psi(y) \leqslant |x-y|^2 \right\}.$$
Optimal solutions of the dual problem are called Kantorovich potentials between $\rho$ and $\mu$. 
If $\rho \in \Paa(\Omega)$, Brenier proves in \cite{B} that the optimal transport plan, $\gamma$, is unique and induced by an optimal transport map, $T$, i.e  $\gamma$ is on the form  $(Id \times T)_{\#} \rho$, where $T_{\#}\rho =\mu$ and $T$ is the gradient of a convex function. 
Moreover, the optimal transport map is given by $T=Id -\nabla \varphi$ where $\varphi$ is a Kantorovich potential between $\rho$ and $\mu$.\\
It is well known that $\Pa(\Omega)$ endowed with the Wasserstein distance defines a metric space and $W_2$ metrizes the narrow convergence of probability measures.
If $\boldsymbol{\rho}=(\rho_1, \dots , \rho_l)$ and $\boldsymbol{\mu}=(\mu_1, \dots , \mu_l)$ are in $\mathcal{P}(\Omega)^l$, we define the product distance by
$$ \boldsymbol{W_2} (\boldsymbol{\rho},\boldsymbol{\mu}) = \left( \sum_{i=1}^l W_2^2(\rho_i, \mu_i) \right)^{1/2}.$$
\subsection{Multi-marginal transportation problem}
In this section, we recall some results from the multi-marginal transport theory that we will used in the sequel. 
We refer, for instance, to \cite{Pass,DMGN} for a complete survey on this topic. 
The usual transport optimal can be extended to several marginals $\rho_1, \dots, \rho_l \in \mathcal{P}(\Omega)$. 
Let $c$ be a cost function from $\Omega^l$ to $\R$, the multi-marginal transport problem, $\W_c$, is defined by
$$ \W_c(\rho_1, \dots , \rho_l):=\inf \left\{ \int_{\Omega^l} c(x_1,\dots,x_l)\,d\lambda(x_1,\dots , x_l) \, :\, \lambda \in \Pi(\rho_1, \dots , \rho_l) \right\},$$
where $\Pi(\rho_1,\dots , \rho_l):= \left\{ \lambda \in \Pa(\Omega^l) \, : \, {\pi^{i}}_{\#} \lambda =\rho_i \right\}$ and $\pi^i$ denotes the canonical projection from $\Omega^l$ to $\Omega$. 
By standard arguments, the existence of an optimal transport plan is guaranteed as in the two marginals case. 
Then, if we assume that $c$ is continuous on $\overline{\Omega}^l$, the following dual formulation holds
$$ \W_c(\rho_1, \dots , \rho_l)=\sup \left\{ \sum_{i=1}^l \int_\Omega u_i(x_i)\,d\rho_i(x_i) \, : \, \sum_{i=1}^l u_i(x_i) \leqslant c(x_1, \dots ,x_l)\right\}.$$
\\
Any optimal $u_1, \dots , u_l$ for the dual formulation are called Kantorovich potentials and are $c$-conjugate functions, i.e
$$ u_i(x_i)= \inf \left\{ c(x_1,\dots , x_l) - \sum_{j=1,j \neq i}^{l} u_j(x_j), \, x_j\in \Omega\right\}, \text{ for all } i=1,\dots ,l.$$ 
\\
For any $\lambda$ optimal transport plan and $u_1, \dots , u_l$ Kantorovich potentials, we get
$$ \sum_{i=1}^l u_i(x_i) = c(x_1, \dots, x_l), \qquad \lambda-a.e.$$
In addition, assuming that $\rho_i$ is absolutely continuous with respect to the Lebesgue measure and $c$ is differentiable in the $i$-th variable, then $u_i$ is a Lipschitz function and
\begin{eqnarray}
\label{eq: u lipschitz}
\nabla u_i(x_i) = \nabla_{x_i} c(x_1, \dots ,x_l),  \qquad \lambda-a.e.
\end{eqnarray}

\section{Assumptions and main result}
\label{part5-paper main result}

In the following, we assume that we have $l>1$ different populations. 
The congestion fonctional associated to the population $\rho_i$ is given by 
$$ \F_i(\rho):= \left\{ \begin{array}{ll}
\int_\Omega F_i(\rho(x)) \, dx & \text{  if } F_i(\rho) \in L^1(\Omega),\\
+\infty & \text{  otherwise. }
\end{array}\right.$$
where $F_i \, : \, \R_+ \rightarrow \R$ is a strictly convex superlinear function of class $\mathcal{C}^2$.
Define $P_i(x):=xF_i'(x)-F_i(x)$ the pressure associated to $F_i$, we assume
\begin{align}
\label{part5-hyp F}
F_i(0)=0 \text{  and  } P_i(x) \leqslant C(1+F_i(x)).
\end{align} 
The typical examples of energies with have in mind are $F(\rho) := \rho \log(\rho)$, which gives a linear diffusion driven by the Laplacian, and $F(\rho):= \rho^m$ ($m>1$), which corresponds to the porous medium diffusion.\\
The multi-marginal interaction energy $\mathcal{W}_i \, : \, \mathcal{P}(\Omega)^l \rightarrow \R$ is defined by
\begin{eqnarray*}
\mathcal{W}_i(\rho_1,\dots , \rho_l):= \inf \left\{ \int_{\Omega^l} c_i(x_1, \dots, x_l) \,d\lambda(x_1, \dots, x_l) \, : \, \lambda\in \Pi(\rho_1,\dots , \rho_l) \right\}.
\end{eqnarray*}
where the cost function $c_i \, : \, \Omega^{l} \rightarrow \mathbb{R}$ is assumed to be continuous on $\overline{\Omega}^l$ and differentiable with respect to $x_i$ such that $\nabla_{x_i} c_i$ is continuous on $\overline{\Omega}^l$ and bounded on $\overline{\Omega}^l$.
\begin{exa}[Barycenter]
Assume $l=3$ and $\rho_1$ evolves minimizing at each step the functional 
$$ (\rho_1, \rho_2, \rho_3) \mapsto \alpha W_{2}^{2}(\rho_1,\rho_2) + \beta W_{2}^2(\rho_1 , \rho_3).$$
That means that $\rho_1$ wants to reach the barycenter in the Wasserstein space of $\rho_2$, $\rho_3$ with weight $\alpha,\beta >0$, see \cite{AC}. 
This functional can be rewritten as the multi-marginal problem
$$ \inf_{\gamma \in \Pi(\rho_1, \rho_2, \rho_3)} \int_{\Omega} c(x,y,z ) \,d\gamma(x,y,z),$$
where $c(x,y,z)= \alpha |x-y|^2 + \beta |x-z|^2$ satisfies the assumptions above.

\end{exa}

The goal of this paper is to study existence and uniqueness of solution to the following nonlinear diffusion system with nonlocal interactions
\begin{eqnarray}
\label{part5-system}
\left\{ \begin{array}{ll}
\partial_t \rho_i = \Delta P_i(\rho_i) +\dive(\rho_i \nabla u_i) & \text{ in } \mathbb{R}^+ \times \Omega,\\
{\rho_i}_{|t=0}=\rho_{i,0}, &
\end{array}\right. , \qquad i \in \{1, \dots ,l\},  
\end{eqnarray}
where $u_i$ is an optimal Kantorovich potential of 
\begin{eqnarray}
\label{eq:multi-marginal problem}
\mathcal{W}_i(\rho_1,\dots , \rho_l):= \inf \left\{ \int_{\Omega^l} c_i(x_1, \dots, x_l) \,d\lambda(x_1, \dots, x_l) \, : \, \lambda\in \Pi(\rho_1,\dots , \rho_l) \right\}.
\end{eqnarray}
Since $\Omega$ is a bounded subset of $\Rn$, \eqref{part5-system} is supplemented with Neumann boundary conditions on $\partial \Omega$,
\begin{eqnarray}
\label{eq:boundary conditions}
(\nabla P_i(\rho_i) +  \nabla u_i \rho_i) \cdot \nu =0 \qquad \text{ on } \mathbb{R}^+ \times \partial \Omega,
\end{eqnarray} 
where $\nu$ is the outward normal to $\partial \Omega$.
To simplify the exposition, we do not treat potentiels or nonlocal interactions in \eqref{eq:multi-marginal problem} even if this can be added easily.
\bigskip

The main difficulty is to handle the nonlinear cross term $\dive(\rho_i \nabla u_i)$. 
However, we remark that if $\lambda_i$ is an optimal transport plan in \eqref{eq:multi-marginal problem} and $\rho_i$ is absolutely continuous with respect to the Lebesgue measure then, by \eqref{eq: u lipschitz},
\begin{eqnarray}
\label{eq:caract pot Kantorovich avec le cout }
\nabla u_i(t,x_i) = \nabla_{x_i} c_i(x_1, \dots,x_l),  \qquad \lambda_i(t)-a.e.
\end{eqnarray}
Consequently, for all $\Phi \in \mathcal{C}_c^{\infty}([0,+\infty) \times \Rn)$,
\begin{multline*}
 \int_0^{+\infty}\int_\Omega \rho_i(t,x) \nabla u_i(t,x) \cdot \nabla \Phi(t,x) \,dxdt\\
  = \int_0^{+\infty}\int_{\Omega^l} \nabla_{x_i} c_i(x_1, \dots, x_l) \cdot \nabla \Phi(t,x_i) \, d\lambda_i(t,x_1, \dots, x_l)dt,
 \end{multline*}
 since $\lambda_i(t)$ solves $\W_i(\rho_1(t), \dots, \rho_l(t))$, $t$-a.e, and \eqref{eq:caract pot Kantorovich avec le cout }. 
Since the right hand side is a linear term with respect to $\lambda_i$, it is easier to work with this one and then, we define a weak solution of \eqref{part5-system}-\eqref{eq:boundary conditions} in the following way.

\begin{de}
A weak solution of \eqref{part5-system}-\eqref{eq:boundary conditions} is a curve $t \in (0,+\infty) \mapsto (\rho_1(t), \dots , \rho_l(t)) \in \Paa(\Omega)^l$ such that $ \nabla P_i(\rho_i) \in \M^n((0,T) \times \Omega)$, for all $T<+\infty$, and
\begin{multline*}
\int_0^{+ \infty} \left( \int_\Omega \partial_t \Phi \rho_i \,dx - \int_\Omega \nabla \Phi \cdot d\nabla P_i(\rho_i) \right)\\
-\int_{\Omega^l} \nabla_{x_i}c_i(x_1, \dots, x_l) \cdot \nabla \Phi(t,x_i) \, d\lambda_i(t,x_1, \dots, x_l)  \, dt
 = - \int_\Omega \Phi(0,x) \rho_{i,0}(x) \, dx,
 \end{multline*}
for every $\Phi \in \mathcal{C}^\infty_c([0,+\infty)\times \Rn)$, where $\lambda_i(t)$ is an optimal transport plan of $\W_i(\rho_1(t), \dots, \rho_l(t))$, $t$-a.e.
\end{de}
Our main result is the following:
\begin{thm}
\label{part5-existence}
If $\rho_{i,0} \in \Paa(\Omega)$ satisfy
\begin{eqnarray}
\label{part5-CI}
\F_i(\rho_{i,0})<+\infty,
\end{eqnarray}
Then \eqref{part5-system}-\eqref{eq:boundary conditions} admits at least one weak solution.
\end{thm}

\begin{rem}
To simplify the analysis we assume that each population has an individual diffusion. 
This implies that solutions are absolutly continuous with respect to the Lebesgue measure and then the Kantorovivh potentials are Lipschitz. 
Theorem \ref{part5-existence} can be generalized replacing $\nabla u_i$ by 
$$ U_i(x_i) = \int_{\Omega^{l-1} }\nabla_{x_i} c_i(x_1, \dots, x_l) \, d\lambda_i^{x_i}(\boldsymbol{\check{x}_i}),$$
where $\boldsymbol{\check{x}_i}=(x_1, \dots, x_{i-1},x_{i+1},\dots ,x_l)$ and $\lambda_i^{x_i}$ is obtained by disintegrating the optimal transport plan $\lambda_i$ with respect to $\rho_i$,
$$ \lambda_i = \lambda_i^{x_i} \otimes \rho_i.$$
\end{rem}

\section{Existence result}
\label{part5-paper semi-implicit scheme}

The proof of Theorem \ref{part5-existence} is based on a variant of the well-known JKO scheme introduced by Jordan, Kinderlherer and Otto, \cite{JKO}. 
We construct by induction with a semi-implicit Euler scheme $l$ sequences $(\rho_{i,h}^k)_{k \in \mathbb{N}} \subset \mathcal{P}^{ac}(\Omega)$, where $h>0$ is a given time step. 
Since the multi-marginal functional $\W_i$ depends on the density $i$, system \eqref{part5-system}-\eqref{eq:boundary conditions} is not a gradient flow in a Wasserstein product space. 
We introduce the functional $\overline{\W}_i(\cdot |\boldsymbol{\mu} )$, where $\boldsymbol{\mu}=(\mu_1, \dots, \mu_l)$, defined by
$$ \overline{\W}_i(\rho|\boldsymbol{\mu}):=\W_i(\mu_1,, \mu_{i-1},\rho,\mu_{i+1}, \dots, \mu_l).$$
In other words, $\overline{\W}_i(\rho|\boldsymbol{\mu})$ is the multi-marginal problem with marginals $\mu_1,\dots, \mu_{i-1},\rho,\mu_{i+1},$ $ \dots, \mu_l$.
\\
Sequences $(\rho_{i,h}^k)_{k \in \mathbb{N}}$ are then constructed using the following semi-implicit JKO scheme: for all $i \in [\![ 1 ,l ]\!]$, $\rho_{i,h}^0=\rho_{i,0}$ and for all $k\geqslant 0$, $\rho_{i,h}^{k+1}$ minimizes
\begin{eqnarray}
\label{eq: JKO scheme}
 \E_{i,h}(\rho | \grho_{h}^{k}) := W_2^2(\rho, \rho_{i,h}^{k}) +2h\left( \F_i(\rho) + \overline{\W}_i(\rho |\grho_{h}^{k} ) \right),
\end{eqnarray}
on $\rho \in \mathcal{P}^{ac}(\Omega)$, where $\grho_{h}^{k}:=(\rho_{1,h}^{k}, \dots, \rho_{l,h}^{k})$. 
At each step, all the marginals are frozen except the $i$-th marginal in the functional \eqref{eq:multi-marginal problem}.\\
These sequences are well defined by standard arguments.
Define the piecewise constant interpolations by, $\rho_{i,h}(0)=\rho_{i,0}$ and for all $t>0$,
\begin{eqnarray}
\label{part5-interpolation}
\rho_{i,h}(t):=\rho_{i,h}^{k+1} \text{ if } t\in(hk,h(k+1)].
\end{eqnarray}
Let $\lambda_{i,h}^{k+1}$ be an optimal transport map for $\W_i\left(\rho_{1,h}^{k}, \dots, \rho_{i-1,h}^{k},\rho_{i,h}^{k+1},\rho_{i+1,h}^{k}, \dots, \rho_{l,h}^{k} \right)$ and $\lambda_{i,h}$ be the piecewise constant interpolation defined by
\begin{eqnarray}
\label{eq:interpolation plan optimal}
 \lambda_{i,h}(t):=\lambda_{i,h}^{k+1} \text{  if } t\in(hk,h(k+1)].
 \end{eqnarray}
\subsection{Basic a priori estimates}

In this section we retrieve the usual estimates in the Wasserstein gradient flow theory. 
First, we show that $\overline{\W}_i$ is Lipschitz in the Wasserstein space.
\begin{lem} 
\label{lem: lipshitz functional}
There exists a constant $C>0$ such that, for all $\boldsymbol{\mu}:=(\mu_1, \dots, \mu_l) \in \Pa(\Omega)^l$, and for all $\rho_1, \rho_2 \in \Paa(\Omega)$,
$$ \overline{\W}_i(\rho_1 | \boldsymbol{\mu} ) - \overline{\W}_i(\rho_2 | \boldsymbol{\mu} ) \leqslant C W_2(\rho_1,\rho_2).$$
\end{lem}

\begin{proof}
Let $\gamma$ be the $W_2$-optimal transport plan between $\rho_1$ and $\rho_2$ and $T$ the $W_2$-optimal transport map associated to $\gamma$ i.e $\gamma =(T \times I)_{\#} \rho_2$. Let $\lambda \in \Pi\left(\mu_1, \dots, \mu_{i-1},\rho_{2},\mu_{i+1}, \dots, \mu_{l}\right)$ optimal for $\overline{\W}_i(\rho_2 | \boldsymbol{\mu} )$. 
Define $\lambda_T$ by
$$ \int_{\Omega^l} \varphi(x_1,\dots,x_l) \, d\lambda_T(x_1,\dots,x_l):= \int_{\Omega^l} \varphi(x_1,\dots,T(x_i),\dots,x_l) \, d\lambda(x_1,\dots,x_l),$$
for all $\varphi \in \mathcal{C}(\Omega^l)$. 
By definition, $\lambda_T \in \Pi\left(\mu_1, \dots, \mu_{i-1},\rho_{1},\mu_{i+1}, \dots, \mu_{l}\right)$. 
Then,
\begin{eqnarray*}
 \overline{\W}_i(\rho_1 | \boldsymbol{\mu} ) - \overline{\W}_i(\rho_2 | \boldsymbol{\mu} ) & \leqslant & \int_{\Omega^l} \left[ c_i(x_1,\dots,T(x_i),\dots,x_l) - c_i(x_1,\dots,x_l)\right] \, d\lambda(x_1,\dots , x_l)\\
& \leqslant & \| \nabla_{x_i} c_i \|_{L^\infty} \int_{\Omega^l} |T(x_i) -x_i |\, d\lambda(x_1,\dots , x_l)\\
&\leqslant & C W_2(\rho_1,\rho_2),
\end{eqnarray*}
where we used the assumption on $\nabla_{x_i} c_i$ and Cauchy-Schwarz inequality.
\end{proof}
In the next proposition, we state usual estimates from JKO scheme.
\begin{prop}
For all $T>0$, there exists $C_T>0$ such that, for all $h,k$, with $hk<T$, $N=\lfloor \frac{T}{h } \rfloor$, for $i \in [\![ 1, l ]\!],$ we have
\begin{eqnarray}
&\F_i(\rho_{i,h}^k) \leqslant C_T,\label{part5-estimation fonctional}\\
&\sum_{k=0}^{N-1} W_2^2(\rho_{i,h}^k,\rho_{i,h}^{k+1}) \leqslant C_Th.\label{part5-estimation distance}
\end{eqnarray}
\end{prop}

\begin{proof}

We first prove \eqref{part5-estimation distance}. 
Since $\rho_{i,h}^{k+1}$ is optimal in the minimization of \eqref{eq: JKO scheme} and $\rho_{i,h}^{k}$ is a competitor, we have
\begin{equation}
\label{part5-variation scheme}
W_2^2(\rho_{i,h}^{k},\rho_{i,h}^{k+1}) \leqslant 2h \left( \F_i(\rho_{i,h}^{k})-\F_i(\rho_{i,h}^{k+1}) +\overline{W}_{i}(\rho_{i,h}^{k} |\grho_{h}^{k})-\overline{W}_{i}(\rho_{i,h}^{k+1} |\grho_{h}^{k}) \right).
\end{equation}
 Then using Lemma \ref{lem: lipshitz functional} in \eqref{part5-variation scheme} and Young's inequality, we obtain
\begin{eqnarray*}
W_2^2(\rho_{i,h}^{k},\rho_{i,h}^{k+1}) &\leqslant & 2h \left( \F_i(\rho_{i,h}^{k})-\F_i(\rho_{i,h}^{k+1})+CW_2(\rho_{i,h}^{k},\rho_{i,h}^{k+1})  \right)\\
&\leqslant & 2h \left( \F_i(\rho_{i,h}^{k})-\F_i(\rho_{i,h}^{k+1}) + \frac{1}{4h}W_2^2(\rho_{i,h}^{k},\rho_{i,h}^{k+1}) + 4C^2h \right).
\end{eqnarray*}
We can thus absorb the $W_2^2$ term in the left-hand side,
$$ \frac{1}{2}W_2^2(\rho_{i,h}^{k},\rho_{i,h}^{k+1}) \leqslant 2h\left( \F_i(\rho_{i,h}^{k})-\F_i(\rho_{i,h}^{k+1})  + Ch  \right).$$
Summing over $k$, we find 
\begin{eqnarray}
\sum_{k=0}^{N-1} W_2^2(\rho_{i,h}^{k},\rho_{i,h}^{k+1}) &\leqslant & 4h\left( \sum_{k=0}^{N-1} (\F_i(\rho_{i,h}^{k})-\F_i(\rho_{i,h}^{k+1}))+C^2T  \right) \nonumber \\
&\leqslant & 4h \left( \F_i(\rho_{i,0}) -\F_i(\rho_{i,h}^{N})+ C^2T  \right). \label{part5-ineq estimation functional}
\end{eqnarray}
Since $\Omega$ is bounded, $\F_i$ is bounded from below and using the assumption \eqref{part5-CI}, we conclude \eqref{part5-estimation distance}. 
The proof is completed noticing that the estimate \eqref{part5-estimation fonctional} comes from \eqref{part5-ineq estimation functional} and \eqref{part5-CI}.

\end{proof}

\subsection{Refined a priori estimates}
\label{part5-paper estimates}

The goal of this section is to obtain stronger estimates on $P_i(\rho_{i,h})$ in order to deal with the nonlinear diffusion term.
\begin{prop}
\label{part5-equality a.e}
For all $i \in [\![ 1,l]\!]$ and for all $k\geqslant 0$, we have $P_i(\rho_{i,h}^{k+1}) \in W^{1,1}(\Omega)$ and
\begin{eqnarray}
\label{part5-equality a.e 1}
h\left(\nabla u_{i,h}^{k+1}\rho_{i,h}^{k+1}+\nabla P_i(\rho_{i,h}^{k+1})\right)=-\nabla \varphi_{i,h}^{k+1} \rho_{i,h}^{k+1}  \qquad \text{      a.e},
\end{eqnarray}
where $\varphi_{i,h}^{k+1}$ is a Kantorovich potential (so that its gradient is unique $\rho_{i,h}^{k+1}-$a.e.) from $\rho_{i,h}^{k+1}$ to $\rho_{i,h}^{k}$ for $W_2$.
\end{prop}

\begin{proof}
The proof is the same as in \cite{A,L} for example. 
We start by taking the first variation in the semi-implicit JKO scheme.
Let $\xi \in \mathcal{C}^\infty_c(\Omega ; \Rn)$ be given and $\Phi_\tau$ the corresponding flow defined by
$$ \partial_\tau \Phi_\tau = \xi \circ \Phi_\tau, \, \Phi_0 =Id.$$
Define the pertubation $\rho_\tau$ of $\rho_{i,h}^{k+1}$ by $\rho_\tau:= {\Phi_\tau}_{\#}\rho_{i,h}^{k+1}$. 
Then we get
\begin{eqnarray}
\label{part5-first variation}
\frac{1}{\tau}\left( \E_{i,h}(\rho_\tau |\grho_{h}^k)-\E_{i,h}(\rho_{i,h}^{k+1} |\grho_{h}^k) \right) \geqslant 0.
\end{eqnarray}
By standard computations, we have
\begin{eqnarray}
\label{part5-first variation dist2}
\limsup_{\tau \searrow 0} \frac{1}{\tau}(W_2^2(\rho_\tau,\rho_{i,h}^{k}) -W_2^2(\rho_{i,h}^{k+1},\rho_{i,h}^{k})) \leqslant \int_{\Omega \times \Omega} (x-y) \cdot \xi(x) \, d\gamma_{i,h}^{k+1}(x,y),
\end{eqnarray}
where $\gamma_{i,h}^{k+1}$ is an $W_2$-optimal transport plan in $\Pi (\rho_{i,h}^{k+1},\rho_{i,h}^{k})$ and $\gamma_{i,h}^{k+1}=(Id \times T_{i,h}^{k+1})_{\#}\rho_{i,h}^{k+1}$ with $T_{i,h}^{k+1} =Id -\nabla \varphi_{i,h}^{k+1}$. 
Moreover, by \eqref{part5-hyp F}, \eqref{part5-estimation fonctional} and Lebesgue's dominated convergence Theorem, we obtain
\begin{eqnarray}
\label{part5-first variation entropie}
\limsup_{\tau \searrow 0} \frac{1}{\tau}(\F_i(\rho_\tau) -\F_i(\rho_{i,h}^{k+1}) )\leqslant - \int_{\Omega } P_i(\rho_{i,h}^{k+1}(x)) \dive(\xi(x))\, dx.
\end{eqnarray}
Finally, by definition of $\lambda_{i,h}^{k+1}$, we have
\begin{eqnarray}
\label{part5-first variation distp}
\limsup_{\tau \searrow 0} \frac{1}{\tau}(\overline{\W}_{i}(\rho_\tau |\grho_{h}^k) -\overline{\W}_{i}(\rho_{i,h}^{k+1}|\grho_{h}^k)) \leqslant \int_{\Omega^l} \nabla_{x_i} c_{i}(x_1,\dots , x_l) \cdot \xi(x_i) \, d\lambda_{i,h}^{k+1}(x_1,\dots , x_l).
\end{eqnarray}
Combining \eqref{eq:caract pot Kantorovich avec le cout }, \eqref{part5-first variation}, \eqref{part5-first variation dist2}, \eqref{part5-first variation entropie}, \eqref{part5-first variation distp}, and replacing $\xi$ by $-\xi$, we find, for all $\xi \in \mathcal{C}_c^\infty(\Omega;\Rn)$,
\begin{equation}
\label{part5-première variation borné}
\int_{\Omega} \nabla \varphi_{i,h}^{k+1} \cdot \xi \rho_{i,h}^{k+1} -h\int_{\Omega} P_i(\rho_{i,h}^{k+1}) \dive(\xi) +h\int_\Omega \nabla u_{i,h}^{k+1}\cdot \xi \rho_{i,h}^{k+1}=0,
\end{equation}
Now we claim that $P_i(\rho_{i,h}^{k+1}) \in W^{1,1}(\Omega)$. 
Indeed, since $P_i$ is controled by $F_i$, \eqref{part5-estimation fonctional} gives $P_i(\rho_{i,h}^{k+1}) \in L^{1}(\Omega)$ and, by \eqref{part5-première variation borné}, we obtain
\begin{eqnarray*}
 \left| \int_\Omega P_i(\rho_{i,h}^{k+1}) \dive(\xi) \right| &\leqslant & \left[ \int_{\Omega} \frac{|\nabla \varphi_{i,h}^k(y)|}{h}\rho_{i,h}^{k+1} +\| \nabla_{x_i} c_i\|_{L^\infty} \right]  \|\xi \|_{L^{\infty}(\Omega)}\\
 & \leqslant & \left[ \frac{W_2(\rho_{i,h}^k,\rho_{i,h}^{k+1})}{h} +C\right] \|\xi\|_{L^{\infty}(\Omega)}.
\end{eqnarray*}
By duality, this implies $P_i(\rho_{i,h}^{k+1}) \in BV(\Omega)$ and $\nabla P_i(\rho_{i,h}^{k+1})=\left(-\nabla u_{i,h}^{k+1}\rho_{i,h}^{k+1}-\frac{\nabla \varphi_{i,h}^{k+1}}{h} \rho_{i,h}^{k+1}\right)$ in $\M^n(\Omega)$.
 In fact, $P_i(\rho_{i,h}^{k+1})$ is in $W^{1,1}(\Omega)$ because $ \nabla u_{i,h}^{k+1}\rho_{i,h}^{k+1} + \frac{\nabla \varphi_{i,h}^k}{h} \rho_{i,h}^{k+1} \in L^1(\Omega)$ and then \eqref{part5-equality a.e 1} is proved.

\end{proof}

We deduce from \eqref{part5-equality a.e 1} an $L^1((0,T),BV(\Omega))$ estimate for $P_i(\rho_{i,h})$.

\begin{cor}
\label{cor:estimate BV}
For all $T>0$, we have
\begin{eqnarray*}
\|P_i(\rho_{i,h})\|_{L^1((0,T);W^{1,1}(\Omega))}\leqslant CT.
\end{eqnarray*} 

\end{cor}

\begin{proof}
Integrating \eqref{part5-equality a.e 1}, we obtain
$$ h\int_\Omega |\nabla P_i(\rho_{i,h}^{k+1})| \leqslant W_2(\rho_{i,h}^{k},\rho_{i,h}^{k+1}) +Ch, $$
Then summing from $k=0$ to $N-1$ and thanks to \eqref{part5-estimation distance}, we have
$$ \int_0^T \int_\Omega |\nabla P_i(\rho_{i,h})| \leqslant CT. $$
We conclude thanks to \eqref{part5-hyp F} and \eqref{part5-estimation fonctional}.
\end{proof}

\subsection{Convergences and proof of Theorem \ref{part5-existence}}
\label{part5-paper convergence}

\subsubsection{Weak and strong convergences of $\rho_{i,h}$}
From the total square distance estimate \eqref{part5-estimation distance}, we deduce the classical $W_2$-convergence,

\begin{prop}
\label{prop:weak convergence}
For all $T >0$ and $i\in\{1,\dots, l\}$, there exists $\rho_i \in \mathcal{C}^{1/2}([0,T],\Paa(\Omega))$ such that, up to extraction of a discrete subsequence,
$$ \sup_{t\in[0,T]} W_2(\rho_{i,h}(t),\rho_i(t)) \rightarrow 0.$$

\end{prop}

\begin{proof}
The proof is classical and is a consequence of \eqref{part5-estimation distance} and a refined version of Arzelà-Ascoli's Theorem \cite[Proposition 3.3.1]{AGS}.
\end{proof}

In order to handle the nonlinear diffusion term, the next proposition proves strong convergence in time and space. 

\begin{prop}
\label{prop:strong convergence}
Up to a subsequence, for all  $i\in\{1,\dots, l\}$, $\rho_{i,h}$ converges strongly in $L^1((0,T)\times \Omega)$ to $\rho_i$ and $\nabla P_i(\rho_{i,h})$ converges narrowly to $\nabla P_i(\rho_i)$.
\end{prop}

\begin{proof}
The proof is now well-known. 
We apply an extension of Aubin-Lions Lemma proved by Rossi and Savaré \cite[Theorem 2]{RS}, see for example \cite{L,CL1}. 
Then we obtain that $\rho_{i,h}$ converges to $\rho_i$ strongly in $L^1((0,T)\times \Omega)$.

It remains to prove that $P_i(\rho_{i,h})$ converges strongly to $P_i(\rho_i)$ in $L^1((0,T) \times \Omega)$.
First, we know that $P_i(\rho_{i,h})$ is uniformly bounded in $L^\infty((0,T),L^1(\Omega))$, using \eqref{part5-hyp F}, and thanks to Corollary \ref{cor:estimate BV}, we have that $P_i(\rho_{i,h})$ is uniformly bounded in $L^1((0,T),W^{1,1}(\Omega))$. 
Then the Sobolev embedding gives that $P_i(\rho_{i,h})$ is uniformly bounded in $L^\infty((0,T),L^1(\Omega)) \cap L^1((0,T),L^{n/n-1}(\Omega))$. 
We deduce that $P_i(\rho_{i,h})$ is uniformly bounded in $L^{(n+1)/n}((0,T)\times \Omega)$, \cite[Lemma 5.3]{CL1}. 
This implies that $P_i(\rho_{i,h})$ is uniformly integrable and Vitali's convergence Theorem gives that $P_i(\rho_{i,h})$ converges strongly to $P_i(\rho_i)$ in $L^1( (0,T) \times \Omega)$. 
Then we conclude the narrow convergence of $\nabla P_i(\rho_{i,h})$ to $\nabla P_i(\rho_i)$ in $\mathcal{M}^n((0,T) \times \Omega)$ thanks to Corollary \ref{cor:estimate BV}.

\end{proof}

\subsubsection{Convergence of $\overline{\W}_i$-optimal transport plans}

First, let us recall some notations. 
Let $\lambda_{i,h}^{k+1}$ be an optimal transport plan for \\
$\W_i\left(\rho_{1,h}^{k}, \dots, \rho_{i-1,h}^{k},\rho_{i,h}^{k+1},\rho_{i+1,h}^{k}, \dots, \rho_{l,h}^{k} \right)$ and $\lambda_{i,h}$ the piecewise constant interpolation of $(\lambda_{i,h}^k)_k$, defined in \eqref{eq:interpolation plan optimal}. \\

In this section, the goal is to prove that $\lambda_{i,h}$ converges to $\lambda_{i}$, where $\lambda_{i}(t)$ is an optimal transport plan for $\W_i\left(\rho_{1}(t), \dots, \rho_{l}(t) \right)$, $t$-a.e.  
To simplify the exposition, we focus on the case $i=1$ and the analysis is similar for $i>1$. 
We introduce the shifted piecewise constant interpolations for all $i \in \{2,\dots , l\}$,
$$ \tilde{\rho}_{i,h}(t) := \rho_{i,h}^k  \text{ if } t \in (hk,h(k+1)] \text{ and }\tilde{\rho}_{i,h}(0):= \rho_{i,0} \text{ if } t=0.$$
and we denote, $\tilde{\grho}_{1,h}$, the $(l-1)$-tuple $(\tilde{\rho}_{2,h} , \dots,\tilde{\rho}_{l,h})$ so that $\lambda_{1,h}(t) \in \Pi \left( \rho_{1,h}(t),\tilde{\grho}_{1,h}(t) \right)$, for all $t>0$.

\begin{prop}
\label{prop: convergence plan transport}
For all $T>0$, $\lambda_{1,h}$ narrowly converges to $\lambda_1$ in $\Pa([0,T] \times  \Omega^{l})$ and $\lambda_1(t) \in \Pi\left( \rho_{1}(t), \dots ,\rho_l(t) \right)$, $t$-a.e.
\end{prop}

\begin{proof}
Proposition \ref{prop:weak convergence} implies that $\tilde{\grho}_{i,h}$ narrowly converges to $\tilde{\grho}_{1}:=( \rho_{2},  \dots, \rho_{l})$ in \\
$L^\infty([0,T],\Paa(\Omega)^{l-1})$.
Define $\lambda_{1,h}^T:= T^{-1}\lambda_{1,h}(t) \otimes dt \in \Pa([0,T] \times \Omega^{l})$. 
Since $[0,T] \times  \Omega^{l}$ is bounded, the sequence $\lambda_{1,h}^T$ is tight then, by Prokhorov's Theorem, $\lambda_{1,h}^T$ narrowly converges to $\lambda_{1}^T$ in $\Pa([0,T] \times \Omega^{l})$.
It remains to show that $\lambda_1^T$ can be written as $T^{-1}\lambda_1(t) \otimes dt$, where $\lambda_1(t) \in \Pi(\rho_1(t) , \dots, \rho_l(t))$ $t$-a.e. 
Denote $\pi^1,\pi^{1,i}$ the projections from $[0,T]\times \Omega^{l}$ to $[0,T]\times \Omega$ and $[0,T]$ with $\pi^{1,i}(t,x_1,\dots ,x_l)=(t,x_i)$, and $\pi^1(t,x_1,\dots ,x_l)=t$. 
Then we have $\pi^{1,1}_{\#}\lambda_{1,h}=\rho_{1,h}(t)dt$, $\pi^{1,i}_{\#}\lambda_{1,h}=\tilde{\rho}_{i,h}(t)dt$, for $i\neq 1$ and $\pi^1_{\#}\lambda_{1,h}=T^{-1}\mathcal{L}_{|[0,T]}$. 
When $h $ goes to $0$, since $\rho_{1,h}(t)dt$ and $\tilde{\rho}_{i,h}(t)dt$ narrowly converge to $\rho_{1}(t)dt$ and $\rho_{i}(t)dt$, we obtain $\pi^{1,i}_{\#}\lambda_{1}=\rho_{i}(t)dt$ and $\pi^1_{\#}\lambda_{i}=T^{-1}\mathcal{L}_{|[0,T]}$, which concludes the proof.

\end{proof}

It remains to prove that the transport plan obtained in the last Proposition \ref{prop: convergence plan transport}, $\lambda_1(t)$, is optimal for $\W_1 \left(\rho_1(t), \dots , \rho_l(t) \right)$.
 We start establishing an approximation result for an optimal transport plan between $\rho_1(t), \rho_2(t), \dots , \rho_l(t)$.

\begin{lem}
\label{part5-approximation transport plan}
Let $\overline{\lambda}_1(t)$ be an optimal transport plan for $\W_1\left(\rho_1(t), \dots , \rho_l(t) \right)$. There exists a sequence of transport plans $\overline{\lambda}_{1,h}(t) \in \Pi(\rho_{1,h}(t),\tilde{\grho}_{1,h}(t))$ such that 
\begin{equation*}
\sup_{t\in[0,T]} W_1(\overline{\lambda}_1(t),\overline{\lambda}_{1,h}(t)) \rightarrow 0.
\end{equation*}
\end{lem}

\begin{proof}
The proof is an adaptation of the one from \cite[Lemma 6.2]{BC1}. 
Let $\gamma_1(t) \in \Pi(\rho_1(t),\rho_{1,h}(t))$ be the optimal transport plan for $W_2$ and, for $i>1$, let $\tilde{\gamma}_i(t) \in \Pi(\rho_i(t),\tilde{\rho}_{i,h}(t))$ be the optimal transport plan for $W_2$.
 Let us disintegrate $\gamma_1(t)$ and $\tilde{\gamma}_i(t)$ as $\gamma_1(t)=\rho_1(t)\otimes\gamma_1^x(t)$ and $\tilde{\gamma}_i(t)=\tilde{\rho}_i(t)\otimes\tilde{\gamma}_i^x(t)$. 
 Now define $\overline{\lambda}_{1,h}(t)$ by
$$ \overline{\lambda}_{1,h}(t)= \int_{\Omega^l}\gamma_1^{x_1}(t)\otimes\tilde{\gamma}_2^{x_2}(t) \otimes \dots \otimes\tilde{\gamma}_l^{x_l}(t) \, d \overline{\lambda}_1(t,x_1, \dots, x_l).$$
By construction, $ \overline{\lambda}_{1,h}(t) \in \Pi(\rho_{1,h}(t),\tilde{\grho}_{1,h}(t))$. 
Then we introduce $\pi$ a transport plan between $\overline{\lambda}_1(t)$ and $\overline{\lambda}_{1,h}(t)$ defined, for all $\varphi \in \mathcal{C}(\Omega^{2l})$, by
$$ \int_{\Omega^{2l}} \varphi(\x,\y) \,d\pi(\x, \y)=\int_{\Omega^l}\left( \int_{\Omega^l}\varphi(\x,\y) \,\gamma_1^{x_1}(t,dy_1)\tilde{\gamma}_2^{x_2}(t,dy_2) \otimes \dots \otimes \tilde{\gamma}_l^{x_l}(t,dy_l)\right) \,\overline{\lambda}_1(t,d\x),$$
where $\x =(x_1, \dots , x_l), \y = (y_1, \dots, y_l)$ are in $\Omega^l$.
Since $\pi \in \Pi(\overline{\lambda}_1(t),\overline{\lambda}_{1,h}(t))$ we have
\begin{eqnarray*}
W_1(\overline{\lambda}_1(t),\overline{\lambda}_{1,h}(t)) & \leqslant &\int_{\Omega^{2l}} (|x_1-y_1|+\sum_{i=2}^l|x_i-y_i|) \,d\pi(\x,\y)\\
&\leqslant & \int_{\Omega^2} |x_1-y_1|\gamma_1^{x_1}(t,dy_1)\rho_1(t,dx_1) +\sum_{i=2}^l\int_{\Omega^2} |x_i-y_i|\tilde{\gamma}_i^{x_i}(t,dy_i)\tilde{\rho}_i(t,dx_i)\\
&\leqslant & \int_{\Omega^2} |x-y|\gamma_1(t,dx,dy)+\sum_{i=2}^l\int_{\Omega^2} |x-y|\tilde{\gamma}_i(t,dx,dy)\\
&\leqslant & W_2^2(\rho_{1,h}(t),\rho_{1}(t))+\sum_{i=2}^l W_2^2(\rho_{i,h}(t),\tilde{\rho}_i(t)) .
\end{eqnarray*}
Then Proposition \ref{prop:weak convergence} concludes the proof.
\end{proof}
~\\
From the previous Lemma, we show that $\lambda_1(t)$ is optimal for $\W_1\left(\rho_1(t), \dots , \rho_l(t) \right)$ $t$-a.e in $[0,T]$.

\begin{prop}
\label{prop:optimal transport plan}
For almost every $t \in [0,T]$, $\lambda_1(t)$ is optimal for $\W_1\left(\rho_1(t), \dots , \rho_l(t) \right)$.
\end{prop}

\begin{proof}
Let $\overline{\lambda}_1(t)$ be an optimal transport plan for $\W_1\left(\rho_1(t), \dots , \rho_l(t) \right)$.
 First, define $\overline{\lambda}_{1,h}(t)$ as in Lemma \ref{part5-approximation transport plan}. 
 Since, by definition, $\lambda_{1,h}$ is optimal for $\W_1(\rho_{1,h},\tilde{\grho}_{1,h})$ and $\overline{\lambda}_{1,h}(t) \in \Pi(\rho_{1,h},\tilde{\grho}_{1,h})$, we have
$$ \int_{\Omega^{l}} c_1(\x) \overline{\lambda}_{1,h}(t,d\x) \geqslant \int_{\Omega^{l}} c_1(\x) \lambda_{1,h}(t,d\x) ,$$
where $\x=(x_1, \dots , x_l)$.
So, for all nonnegative function $\varphi \in \mathcal{C}^\infty([0,T])$, we get
$$\int_0^T \int_{ \Omega^{l}} c_1(\x) \overline{\lambda}_{1,h}(t,d\x) \varphi(t) \,dt\geqslant \int_0^T\int_{\Omega^{l}} c_1(\x) \lambda_{1,h}(t,d\x) \varphi(t) \,dt.$$
Since $\Omega$ is bounded and according to Lemma \ref{part5-approximation transport plan}, 
$$\int_0^T \int_{\Omega^{l}} c_1(\x) \overline{\lambda}_{1,h}(t,d\x) \varphi(t) \,dt\rightarrow\int_0^T \int_{ \Omega^{l}} c_1(\x) \overline{\lambda}_{1}(t,d\x) \varphi(t) \,dt,$$
as $h\rightarrow 0$. 
In addition, since $\lambda_{1,h}$ narrowly converges to $\lambda_{1}$ in $\Pa([0,T] \times  \Omega^{l})$, we have
$$\int_0^T\int_{\Omega^{l}} c_1(\x) \lambda_{1,h}(t,d\x) \varphi(t) \,dt\rightarrow\int_0^T\int_{ \Omega^{l}} c_1(\x) \lambda_{1}(t,d\x) \varphi(t) \,dt.$$
And then 
$$\int_0^T \int_{\Omega^{l}} c_1(\x) \overline{\lambda}_{1}(t,d\x) \varphi(t) \,dt \geqslant \int_0^T\int_{\Omega^{l}} c_1(\x) \lambda_{1}(t,d\x) \varphi(t) \,dt.$$
The inequality holds for all nonnegative function $\varphi \in \mathcal{C}^\infty([0,T])$, we thus obtain, for almost every $t \in [0,T]$,
$$ \W_{1}\left(\rho_1(t), \dots , \rho_l(t) \right) \geqslant \int_{\Omega^{l}}c_1(\x) \lambda_1(t,d\x),$$
and the proof is concluded.
\end{proof}

\subsubsection{Proof of Theorem \ref{part5-existence}}

First, we show that $(\rho_{1,h}, \dots , \rho_{l,h})$ is solution of a discrete approximation of system \eqref{part5-system}.
\begin{prop}
\label{part5-discreet system}
Let $h>0$, for all $T>0$, let $N$ such that $Nh=T$ and for all $\phi \in \mathcal{C}^\infty_c ([0,T)\times \Omega)$, then
\begin{eqnarray*}
\int_0^T \int_{\Omega} \rho_{i,h}(t,x)  \partial_t \phi(t,x) \,dxdt&=&h\sum_{k=0}^{N-1}\int_{\Omega} \nabla P_i(\rho_{i,h}^{k+1}(x)) \cdot \nabla \phi (t_{k},x)\,dx\\
&+&h\sum_{k=0}^{N-1}\int_{\Omega^l} \nabla_{x_i} c_i(\x) \cdot  \nabla \phi(t_{k},x_i) \, d\lambda_{i,h}^{k+1}(\x)\\
&+&\sum_{k=0}^{N-1}\int_{\Omega \times \Omega} \mathcal{R}[\phi(t_{k},\cdot)](x,y) d\gamma_{i,h}^{k+1} (x,y)\\
&-& \int_{\Omega} \rho_{i,0}(x) \phi(0,x) \, dx,
\end{eqnarray*}
with, for all $\phi \in \mathcal{C}^\infty_c([0,T) \times \Rn)$,
$$ |\mathcal{R}[\phi](x,y)| \leqslant \frac{1}{2} \|D^2 \phi \|_{L^\infty ([0,T) \times \Omega)} |x- y|^2,$$
$\gamma_{i,h}^{k+1}$ is an optimal transport plan in $\Gamma (\rho_{i,h}^{k},\rho_{i,h}^{k+1})$.

\end{prop}

\begin{proof}
This is a consequence of \eqref{part5-equality a.e 1} (see \cite{A,L}).
\end{proof}
Now, we have to take the limit in the system of Proposition \ref{part5-discreet system}. 
The linear term (with time derivative) and the diffusion term converge to the desired result thanks to Proposition \ref{prop:strong convergence}.
 The remainder term goes to $0$ as $h$ goes to $0$ because of \eqref{part5-estimation distance}. 
 So it remains to check the convergence of multi-marginal interaction terms. 
 By Proposition \ref{prop: convergence plan transport}, $\lambda_{i,h}$ converges to $\lambda_i$ in $\Pa([0,T] \times \Omega^l)$ and then,
$$ h\sum_{k=0}^{N-1}\int_{\Omega^l} \nabla_{x_i} c_{i}(\x) \cdot  \nabla \phi(t_{k},x_i) \,d\lambda_{i,h}^{k+1}(\x) \rightarrow \int_0^T \int_{\Omega^l} \nabla_{x_i} c_{i}(\x) \cdot \nabla \phi(t,x_i) \,d\lambda_{i}(t,\x)dt.$$
and, by Proposition \ref{prop:optimal transport plan}, $\lambda_{i}(t)$ is an optimal transport plan for $\W_i(\rho_1,\dots ,\rho_l)$.

\section{Uniqueness in dimension one}
\label{part5-paper uniqueness}

In this section, $\Omega$ is a compact convex subset in $\R$. 
We give an uniqueness result based on a displacement convexity argument and some examples of functionals satisfying this condition. 
Although Theorem \ref{part5-uniqueness} holds in dimension higher than one, we retrict ourselves to the dimension one because as far as we know, there is no example of multi-marginal functional geodesicaly convex in higher dimension.

\subsection{Displacement convexity in product Wasserstein space}

For the purpose of this paper, it is enough to restrict ourselves to absolutely continuous probability measures.
 Given $\rho_0$ and $\rho_1$ in $\mathcal{P}^{ac}(\Omega)$, there exists a unique optimal transport map $T$ between $\rho_0$ and $\rho_1$ i.e $T_{\#} \rho_0 = \rho_1$ and 
$$ W_2^2(\rho_0 ,\rho_1) = \int_\Omega |x -T(x) |^2 \rho_0(x) \,dx.$$
In addition, $T \, : \, \Omega \rightarrow \Omega$ is a nondecreasing map.\\ 
The Wasserstein geodesic between $\rho_0$ and $\rho_1$ is the curve $t \in [0,1] \mapsto \rho_t$ given by the McCann's interpolation
$$ \rho_t := {T^t}_{\#} \rho_0,$$
where $T^t = (1-t) Id + tT$ is the optimal transport map between $\rho_0$ and $\rho_t$, and $ \rho_t$ is a constant speed geodesic:
$$ W_2(\rho_t,\rho_s) = |t-s| W_2(\rho_0,\rho_1).$$
Now we recall the definition of geodesically convex functional in Wasserstein product space.

\begin{de}
Let $\lambda \in \R$. A functional $\W \, : \, \Pa(\Omega)^l \rightarrow (-\infty,+\infty]$ is said $\lambda$-geodesically convex in $\Pa(\Omega)^l$ if for every $i \in [\![ 1,l]\!]$ and for every couple $(\mu_i^0,\mu_i^1) \in \Pa(\Omega)^2$ 
$$ \W(\mu_1^t,\dots, \mu_l^t) \leqslant (1-t) \W(\mu_1^0,\dots, \mu_l^0)+t\W(\mu_1^1,\dots, \mu_l^1) -\frac{\lambda}{2}t(1-t)\boldsymbol{W_2^2}((\mu_1^0,\dots, \mu_l^0),(\mu_1^1,\dots, \mu_l^1)),$$
where $\mu_i^t$ is a constant speed geodesic between $\mu_i^0$ and $\mu_i^1$ and $\boldsymbol{W_2}$ is the product distance on $\Pa(\Omega)^l$.
\end{de}

\smallskip 

Note that if $F \, :\, [0,+\infty) \rightarrow \R$ satisfies McCann's condition i.e. the map 
\begin{equation}
\label{part5-McCann condition}
r \in (0,+\infty) \mapsto r^{n}F(r^{-n}) \text{ is convex nonincreasing},
\end{equation}
then it is well-known that  
$$ \F(\rho):= \left\{ \begin{array}{ll}
\int_\Omega F(\rho) & \text{ if }F(\rho) \in L^1(\Omega),\\
+\infty & \text{ otherwise},
\end{array}\right.$$
 is geodesically convex ($\lambda=0$), see \cite{MC}.

\bigskip
In the following we give a class of multi-marginal functionals geodesically convex. First, we provide a characterization of the co-monotone transport plan as in \cite[Lemma 2.8]{S}.

\begin{lem}
\label{lem:caract optimal transport plan}
For $l \geqslant 2$, let $\gamma$ be a transport plan having $\rho_1, \dots, \rho_l$ as marginals.
 If $\gamma$ satisfies the property
\begin{eqnarray}
\label{eq:monotony transport plan}
(x_1,\dots , x_l),(y_1, \dots , y_l) \in \supp \gamma \Rightarrow \left[ x_1 < y_1 \Rightarrow \forall i, \, x_i \leqslant y_i \right],
\end{eqnarray}
then $\gamma = \gamma_{mon}:= (G_1^{[-1]}, \dots, G_l^{[-1]})_{\#} \mathcal{L}^1_{|[0,1]}$, where $G_i^{[-1]}$ is the pseudo-inverse of the cumulative distribution function of $\rho_i$, $G_i(a)=\rho_i((-\infty, a])$.
\end{lem}

\begin{proof}
This lemma is an extension of \cite[Lemma 2.8]{S} (where the case $l=2$ is studied) and the proof is similar.
 First, for all $a_1, \dots , a_l \in \R$, we know that $\gamma_{mon}((-\infty, a_1] \times \dots \times (-\infty , a_l]) = \min_i G_i(a_i)$. 
 Indeed, by definition of $\gamma_{mon} \in \Pi(\rho_1, \dots, \rho_l)$,
\begin{eqnarray*}
\gamma_{mon}((-\infty, a_1] \times \dots \times (-\infty , a_l]) &=& \left|\{ x \in [0,1] \, : \, \forall i, \, G_i^{[-1]}(x) \leqslant a_i \} \right|\\
&=& \left|\{ x \in [0,1] \, : \, \forall i, \, x \leqslant G_i(a_i) \} \right|\\
&=& \min_i G_i(a_i).
\end{eqnarray*}
Since the knowledge of $\gamma((-\infty, a_1] \times \dots \times (-\infty , a_l])$, for all $a_1, \dots , a_l \in \R$ is enough to characterize $\gamma$, we just need to show that 
$$ \gamma((-\infty, a_1] \times \dots \times (-\infty , a_l])=\min_i G_i(a_i),  $$
to conclude the proof. Define for all $i$, the set $A_i= \Pi_{j=1}^{i-1}[a_j, +\infty) \times (-\infty , a_i] \times \Pi_{j=i+1}^l [a_j, +\infty)$. 
Since $\gamma$ satisfies \eqref{eq:monotony transport plan}, for all $i \neq j$ we cannot have both $\gamma (A_i) >0$ and $\gamma (A_j) >0$. 
Then,
\begin{eqnarray*}
\gamma((-\infty, a_1] \times \dots \times (-\infty , a_l]) & =& \min_i \gamma((-\infty, a_1] \times \dots \times (-\infty , a_l] \cup A_i)\\
&=& \min_i \gamma( \R \times \dots \times \R \times (-\infty, a_i] \times \R \dots \times \R)\\
&=& \min_i \rho_i ((-\infty, a_i])\\
&=& \min_i G_i(a_i).
\end{eqnarray*}
\end{proof}

~\\
This lemma allows us to study the geodesic convexity of multi-marginal functionals for a large class of costs.

\begin{prop}
\label{prop:geo conv multimarge}
Let $c \, : \, \Omega^l \rightarrow \R$ be a $\mathcal{C}^2$ convex function satisfying $\partial_{i,j} c \leqslant 0$ for all $i\neq j$. 
The functional $\W_c \, : \, \Paa(\Omega)^l \rightarrow \R$ defined by
$$ \W_c(\rho_1, \dots , \rho_l) := \inf\left\{ \int_{\Omega^l} c(x_1, \dots , x_l) \, d\gamma (x_1, \dots , x_l) \, : \, \gamma \in \Pi(\rho_1 , \dots , \rho_l) \right\}, $$
is geodesically convex in $\Paa(\Omega)^l$.

\end{prop}

\begin{proof}
Given $(\rho_1^0, \dots , \rho_l^0)$ and $(\rho_1^1, \dots , \rho_l^1)$ in $\mathcal{P}^{ac}(\Omega)^l$, define the constant speed geodesic between $\rho_i^0$ and $\rho_i^1$, $\rho_i^t={T_i^t}_{\#} \rho_i^0$. 
Let $\gamma_0$ be an optimal transport plan for the multi-marginal problem $\W_c( \rho_1^0, \dots , \rho_l^0)$. 
By \cite[Theorem 4.1]{C_multi}, there exist $l-1$ nondecreasing maps $S_2, \dots , S_l$ such that $\gamma_0 = (Id, S_2, \dots , S_l)_{\#} \rho_1^0$. \\
Define the interpolation plan $\gamma_t$ by 
$$ \gamma_t= (T_1^t, \dots , T_l^t)_{\#} \gamma_0 = (T_1^t, T_2^t \circ S_2, \dots , T_l^t \circ S_l)_{\#} \rho_1^0.$$
Observe that $\gamma_1 = (T_1, T_2 \circ S_2, \dots , T_l \circ S_l)_{\#} \rho_1^0$ and since $T_1$ and for all $i\geqslant 2$, $ T_i \circ S_i$ are nondecreasing maps, $ \gamma_1$ satisfies \eqref{eq:monotony transport plan}. 
We want to show that $\gamma_1$ is an optimal transport plan for $\W_c(\rho_1^1, \dots , \rho_l^1)$. 
Theorem 4.1 from \cite{C_multi} says that the optimal transport plan $\gamma_{opt}$ of $\W_c(\rho_1^1, \dots , \rho_l^1)$ is of the form $\gamma_{opt} = (Id, \tilde{S}_2, \dots , \tilde{S}_l)_{\#} \rho_1^1$, where $\tilde{S}_i$ is nondecreasing. 
Then $\gamma_{opt}$ also satisfies \eqref{eq:monotony transport plan} and then by Lemma \ref{lem:caract optimal transport plan}, we conclude that $\gamma_1$ is an optimal transport plan for $\W_c(\rho_1^1, \dots , \rho_l^1)$.\\
By convexity of $c$, we have
\begin{eqnarray*}
\W_c(\rho_1^t, \dots , \rho_l^t) & \leqslant & \int_{\Omega^l} c(x_1, \dots , x_l) \,d\gamma_t\\
&\leqslant & \int_{\Omega^l} c(T_1^t(x_1), \dots , T_l^t(x_l)) \,d\gamma_0 \\
&\leqslant & (1-t)\int_{\Omega^l} c(x_1, \dots , x_l) \,d\gamma_0 + t\int_{\Omega^l} c(T_1(x_1), \dots , T_l(x_l)) \,d\gamma_0\\
&\leqslant & (1-t) \W_c(\rho_1^0, \dots , \rho_l^0) +t\W_c(\rho_1^1, \dots , \rho_l^1),
\end{eqnarray*}
which concludes the proof.

\end{proof}

\begin{rem} 
This result cannot be generalized in higher dimension. Indeed, in dimension $n>1$, it is well known that $W_2(\cdot, \sigma)$ is not $\lambda$-convex along geodesic on $\mathcal{P}(\Omega)$ (see example 9.1.5 from \cite{AGS}). 
\end{rem}

\subsection{Wasserstein contraction}

First, let us define the Fréchet subdifferential for $\W \, : \, \Paa(\Omega)^l \rightarrow (-\infty,+\infty]$ by extending the definition given in \cite{AGS}.

\begin{de}
Let $\W \, : \, \Paa(\Omega)^l \rightarrow (-\infty,+\infty]$ be a functional and let $\boldsymbol{\xi}=(\xi_1, \dots,\xi_l) \in L^2((\mu_1,\dots,\mu_l),\Omega)$, i.e
$$ \sum_{i=1}^l \int_\Omega |\xi_i|^2 \mu_i  <+\infty.$$
We say that $\boldsymbol{\xi}$ is in the Fréchet subdifferential $\partial \W(\mu_1,\dots,\mu_l)$ if
\begin{equation}
\label{part5-Frechet subdif}
\liminf_{\boldsymbol{\nu} \rightarrow \boldsymbol{\mu} } \frac{\W(\boldsymbol{\nu})-\W(\boldsymbol{\mu}) -\sum_{i=1}^l \int_\Omega \langle \xi_i(x) , T_{\mu_i}^{\nu_i}(x) -x \rangle \mu_i(x) \, dx }{ \boldsymbol{W_2}(\boldsymbol{\nu},\boldsymbol{\mu}) } \geqslant 0,
\end{equation}
where $\boldsymbol{\mu}:=(\mu_1, \dots , \mu_l)$ and $T_{\mu_i}^{\nu_i}$ is the optimal transport map between $\mu_i$ and $\nu_i$.
\end{de}

The next proposition characterizes the subdifferential of $\lambda$-geodesically convex functionals.

\begin{prop}
Let $\W \, : \, \Paa(\Omega)^l \rightarrow (-\infty,+\infty]$ be a $\lambda$-geodesically convex functional.
 Then a vector $\boldsymbol{\xi} \in L^2(\boldsymbol{\mu},\Omega)$ belongs to the Fréchet subdifferential of $\W$ at $\boldsymbol{\mu}$ if and only if
\begin{eqnarray}
\label{part5-caracterisation subdif 1}
\W(\boldsymbol{\nu})-\W(\boldsymbol{\mu}) \geqslant \sum_{i=1}^l \int_\Omega \langle \xi_i (x), T_{\mu_i}^{\nu_i}(x)-x \rangle \mu_i(x) \, dx + \frac{\lambda}{2}\boldsymbol{W_2^2}(\boldsymbol{\nu},\boldsymbol{\mu}),
\end{eqnarray}
for all $\boldsymbol{\nu} \in \Paa(\Omega)^l$.
Moreover, if $\boldsymbol{\xi} \in \partial\W(\boldsymbol{\mu})$ and $\boldsymbol{\kappa} \in \partial\W(\boldsymbol{\nu})$ then
\begin{eqnarray}
\label{part5-caracterisation subdif 2}
\sum_{i=1}^l\int_\Omega \langle \xi_i (x) - \kappa_i(T_{\mu_i}^{\nu_i}(x)), T_{\mu_i}^{\nu_i}(x)-x \rangle \mu_i(x) \, dx \leqslant - \lambda \boldsymbol{W_2^2}(\boldsymbol{\mu},\boldsymbol{\nu}).
\end{eqnarray}
\end{prop}

\begin{proof}
The proof is the same as in the characterization by Variational inequalities and monotonicity done in \cite{AGS} p. 231.

\end{proof}

We can now prove the following uniqueness result.

\begin{thm}
\label{part5-uniqueness}
Assume $F_{i}$ satisfies \eqref{part5-McCann condition} and $\W_i$ is a $\lambda_i$-geodesically convex functional. 
Let $\boldsymbol{\rho}^1:=(\rho_1^1, \dots, \rho_l^1)$ and $\boldsymbol{\rho}^2:=(\rho_1^2, \dots, \rho_l^2)$, in $\mathcal{P}^{ac}(\Omega)$, two weak solutions of \eqref{part5-system} with initial conditions $\rho_i^1(0, \cdot)= \rho_{i,0}^1$ and $\rho_i^2(0, \cdot)= \rho_{i,0}^2$. 
If for all $T<+\infty$,
\begin{equation}
\label{part5-hyp gradient field}
\int_0^T \sum_{i=1}^l \|v^1_{i,t}\|_{L^2(\rho_{i,t}^1)} \,dt + \int_0^T \sum_{i=1}^l \|v^2_{i,t}\|_{L^2(\rho_{i,t}^2)} \,dt <+\infty,
\end{equation}
with, for $j\in \{1,2\}$,
 $$ v^j_{i,t}:= - \nabla F_i'(\rho_{i,t}^j) -\nabla u_i^j,$$
then for every $t\in   [0 ,T ]$,
$$  \boldsymbol{W_2^2}(\boldsymbol{\rho}_t^1,\boldsymbol{\rho}_t^2) \leqslant e^{-\left(2\sum_{i=1}^l \lambda_i\right) t}\boldsymbol{W_2^2}(\boldsymbol{\rho}^1_0,\boldsymbol{\rho}^2_0),$$
\end{thm}

\begin{proof}
Using Theorem 5.24, Corollary 5.25 from \cite{S} and assumption \eqref{part5-hyp gradient field}, we obtain
$$ \frac{d}{dt}\left(\frac{1}{2}W_2^2(\rho_{i,t}^1,\rho_{i,t}^2)\right) = \int_\Omega \langle x - T_{i}^t(x) , v^1_{i,t}(x) -v^2_{i,t}(T_{i}^t(x)) \rangle \rho_i^1(x) \, dx,$$
where $T_{i}^t$ is the optimal transport map between $\rho_{i,t}^1$ and $\rho_{i,t}^2$. 
Since $F_i$ satisfies McCann's condition, we have
$$ \int_{\Omega} \langle x - T_{i}^t(x) ,  \nabla F'_i(\rho_{i,t}^2  (T_{i}^t(x)))-\nabla F'_i(\rho_{i,t}^1 (x) ) \rangle \rho_{i,t}^1(x) \, dx \leqslant 0.$$  
In addition, $\W_i$ is $\lambda_i$-geodesically convex then \eqref{part5-caracterisation subdif 2} gives
$$
\sum_{i=1}^l\int_\Omega \langle \nabla u_{i,t}^1 (x) - \nabla u_{i,t}^2(T_{i}^t(x)), T_{i}^t(x)-x \rangle \rho_{i,t}^1(x) \, dx \leqslant - \sum_{i=1}^l\lambda_i \boldsymbol{W_2^2}(\boldsymbol{\rho}_{t}^1,\boldsymbol{\rho}_{t}^2).$$
Summing over $i$ and combining these inequalities, we obtain
$$ \frac{d}{dt}\left( \boldsymbol{W_2^2}(\boldsymbol{\rho}_{t}^1,\boldsymbol{\rho}_{t}^2) \right) \leqslant - \left(2 \sum_{i=1}^l\lambda_i \right)  \boldsymbol{W_2^2}(\boldsymbol{\rho}_{t}^1,\boldsymbol{\rho}_{t}^2).$$
Gronwall's Lemma concludes the proof.
\end{proof}

\begin{rem}
Assumption \eqref{part5-hyp gradient field} in Theorem \ref{part5-uniqueness} is made to ensure the absolute continuity of $W_2 (\rho_{i,t}^1,\rho_{i,t}^2)$ and can be checked using \eqref{part5-equality a.e 1} (see for example \cite[Proposition 7.3]{L}).

\end{rem}

 \bibliographystyle{plain}
\bibliography{Laborde_bib}

\begin{thebibliography}{10}

\bibitem{A}
M.~Agueh.
\newblock Existence of solutions to degenerate parabolic equations via the
  {M}onge-{K}antorovich theory.
\newblock {\em Adv. Differential Equations}, 10(3):309--360, 2005.

\bibitem{AC}
M.~Agueh and G.~Carlier.
\newblock Barycenters in the {W}asserstein space.
\newblock {\em SIAM J. Math. Anal.}, 43(2):904--924, 2011.

\bibitem{AGS}
L.~Ambrosio, N.~Gigli, and G.~Savar{\'e}.
\newblock {\em Gradient flows in metric spaces and in the space of probability
  measures}.
\newblock Lectures in Mathematics ETH Z\"{u}rich. Birkh\"{a}user Verlag, Basel,
  2005.

\bibitem{BC1}
J.-B. Baillon and G.~Carlier.
\newblock From discrete to continuous {W}ardrop equilibria.
\newblock {\em Netw. Heterog. Media}, 7(2):219--241, 2012.

\bibitem{B}
Y.~Brenier.
\newblock Polar factorization and monotone rearrangement of vector-valued
  functions.
\newblock {\em Comm. Pure Appl. Math.}, 44(4):375--417, 1991.

\bibitem{BS}
G.~Buttazzo and F.~Santambrogio.
\newblock A model for the optimal planning of an urban area.
\newblock {\em SIAM J. Math. Anal.}, 37(2):514--530, 2005.

\bibitem{BS1}
G.~Buttazzo and F.~Santambrogio.
\newblock A mass transportation model for the optimal planning of an urban
  region.
\newblock {\em SIAM Rev.}, 51(3):593--610, 2009.

\bibitem{C_multi}
G.~Carlier.
\newblock On a class of multidimensional optimal transportation problems.
\newblock {\em J. Convex Anal.}, 10(2):517--529, 2003.

\bibitem{CE}
G.~Carlier and I.~Ekeland.
\newblock The structure of cities.
\newblock {\em J. Global Optim.}, 29(4):371--376, 2004.

\bibitem{CE1}
G.~Carlier and I.~Ekeland.
\newblock Equilibrium structure of a bidimensional asymmetric city.
\newblock {\em Nonlinear Anal. Real World Appl.}, 8(3):725--748, 2007.

\bibitem{CL1}
G.~Carlier and M.~Laborde.
\newblock A splitting method for nonlinear diffusions with nonlocal,
  nonpotential drifts.
\newblock {\em Nonlinear Analysis: Theory, Methods \& Applications}, 150:1 --
  18, 2017.

\bibitem{CS}
G.~Carlier and F.~Santambrogio.
\newblock A variational model for urban planning with traffic congestion.
\newblock {\em ESAIM Control Optim. Calc. Var.}, 11(4):595--613, 2005.

\bibitem{DFF}
M.~Di~Francesco and S.~Fagioli.
\newblock Measure solutions for non-local interaction {PDE}s with two species.
\newblock {\em Nonlinearity}, 26(10):2777--2808, 2013.

\bibitem{DMGN}
S.~Di~Marino, A.~Gerolin, and L.~Nenna.
\newblock Optimal transportation theory with repulsive costs.
\newblock In {\em Topological optimization and optimal transport}, volume~17 of
  {\em Radon Ser. Comput. Appl. Math.}, pages 204--256. De Gruyter, Berlin,
  2017.

\bibitem{JKO}
R.~Jordan, D.~Kinderlehrer, and F.~Otto.
\newblock The variational formulation of the {F}okker-{P}lanck equation.
\newblock {\em SIAM J. Math. Anal.}, 29(1):1--17, 1998.

\bibitem{KMX}
D.~Kinderlehrer, L.~Monsaingeon, and X.~Xu.
\newblock A {W}asserstein gradient flow approach to {P}oisson-{N}ernst-{P}lanck
  equations.
\newblock {\em ESAIM Control Optim. Calc. Var.}, 23(1):137--164, 2017.

\bibitem{L}
M.~Laborde.
\newblock On some nonlinear evolution systems which are perturbations of
  {W}asserstein gradient flows.
\newblock In {\em Topological optimization and optimal transport}, volume~17 of
  {\em Radon Ser. Comput. Appl. Math.}, pages 304--332. De Gruyter, Berlin,
  2017.

\bibitem{MC}
R.~J. McCann.
\newblock A convexity principle for interacting gases.
\newblock {\em Adv. Math.}, 128(1):153--179, 1997.

\bibitem{Pass}
B.~Pass.
\newblock Multi-marginal optimal transport: theory and applications.
\newblock {\em ESAIM Math. Model. Numer. Anal.}, 49(6):1771--1790, 2015.

\bibitem{RS}
R.~Rossi and G.~Savar{\'e}.
\newblock Tightness, integral equicontinuity and compactness for evolution
  problems in {B}anach spaces.
\newblock {\em Ann. Sc. Norm. Super. Pisa Cl. Sci. (5)}, 2, 2003.

\bibitem{S1}
F.~Santambrogio.
\newblock Transport and concentration problems with interaction effects.
\newblock {\em J. Global Optim.}, 38(1):129--141, 2007.

\bibitem{S2}
F.~Santambrogio.
\newblock {\em Variational problems in transport theory with mass
  concentration}, volume~4 of {\em Tesi. Scuola Normale Superiore di Pisa
  (Nuova Series) [Theses of Scuola Normale Superiore di Pisa (New Series)]}.
\newblock Edizioni della Normale, Pisa, 2007.
\newblock Dissertation, Scuola Normale Superiore, Pisa, 2006.

\bibitem{S_grenoble}
F.~Santambrogio.
\newblock Models and applications of optimal transport in economics, traffic,
  and urban planning.
\newblock In {\em Optimal transportation}, volume 413 of {\em London Math. Soc.
  Lecture Note Ser.}, pages 22--40. Cambridge Univ. Press, Cambridge, 2014.

\bibitem{S}
F.~Santambrogio.
\newblock {\em Optimal Transport for Applied Mathematicians}.
\newblock Progress in Nonlinear Differential Equations and Their Applications
  87. Birkasauser Verlag, Basel, 2015.

\bibitem{V1}
C.~Villani.
\newblock {\em Topics in optimal transportation}, volume~58 of {\em Graduate
  Studies in Mathematics}.
\newblock American Mathematical Society, Providence, RI, 2003.

\bibitem{V2}
C.~Villani.
\newblock {\em Optimal transport}, volume 338 of {\em Grundlehren der
  Mathematischen Wissenschaften [Fundamental Principles of Mathematical
  Sciences]}.
\newblock Springer-Verlag, Berlin, 2009.
\newblock Old and new.

\end{thebibliography}
\end{document}